\pdfoutput=1
\RequirePackage{ifpdf}
\ifpdf 
\documentclass[pdftex]{sigma}
\else
\documentclass{sigma}
\fi

\usepackage{calrsfs}

\newcommand{\F}{{\mathbb{F}}}
\newcommand{\Z}{{\mathbb{Z}}}
\newcommand{\Q}{{\mathbb{Q}}}
\newcommand{\N}{{\mathbb{N}}}

\newcommand{\bG}{{\mathbf{G}}}
\newcommand{\bB}{{\mathbf{B}}}
\newcommand{\bT}{{\mathbf{T}}}
\newcommand{\bU}{{\mathbf{U}}}

\newcommand{\cS}{{\mathcal{S}}}

\newcommand{\fb}{{\underline{\mathfrak{b}}}}
\newcommand{\fe}{{\mathfrak{e}}}

\newcommand{\fu}{{\underline{\mathfrak{u}}}}

\newcommand{\Hom}{{\operatorname{Hom}}}
\newcommand{\Ind}{{\operatorname{Ind}}}

\newcommand{\St}{{\operatorname{St}}}

\newcommand{\GL}{{\operatorname{GL}}}

\renewcommand{\geq}{\geqslant}

\newtheorem{thm}{Theorem}[section]
\newtheorem{cor}[thm]{Corollary}
\newtheorem{lem}[thm]{Lemma}
\newtheorem{prop}[thm]{Proposition}
\theoremstyle{definition}

\newtheorem{exmp}[thm]{Example}
\newtheorem{rem}[thm]{Remark}
\numberwithin{equation}{section}

\begin{document}
\allowdisplaybreaks

\newcommand{\arXivNumber}{1708.07782}

\renewcommand{\thefootnote}{}

\renewcommand{\PaperNumber}{091}

\FirstPageHeading

\ShortArticleName{James' Submodule Theorem and the Steinberg Module}
\ArticleName{James' Submodule Theorem\\ and the Steinberg Module\footnote{This paper is a~contribution to the Special Issue on the Representation Theory of the Symmetric Groups and Related Topics. The full collection is available at \href{https://www.emis.de/journals/SIGMA/symmetric-groups-2018.html}{https://www.emis.de/journals/SIGMA/symmetric-groups-2018.html}}}
\Author{Meinolf GECK}
\AuthorNameForHeading{M.~Geck}
\Address{IAZ - Lehrstuhl f\"ur Algebra, Universit\"at Stuttgart,\\
Pfaffenwaldring 57, D--70569 Stuttgart, Germany}
\Email{\href{mailto:meinolf.geck@mathematik.uni-stuttgart.de}{meinolf.geck@mathematik.uni-stuttgart.de}}
\URLaddress{\url{http://www.mathematik.uni-stuttgart.de/~geckmf/}}

\ArticleDates{Received August 29, 2017, in f\/inal form November 28, 2017; Published online December 05, 2017}

\Abstract{James' submodule theorem is a fundamental result in the representation theory of the symmetric groups and the f\/inite general linear groups. In this note we consider a~version of that theorem for a general f\/inite group with a split $BN$-pair. This gives rise to a distinguished composition factor of the Steinberg module, f\/irst described by Hiss via a~somewhat dif\/ferent method. It is a~major open problem to determine the dimension of this composition factor.}


\Classification{20C33; 20C20}

\renewcommand{\thefootnote}{\arabic{footnote}}
\setcounter{footnote}{0}

\section{Introduction} \label{sec0}
Let $G$ be a f\/inite group of Lie type and $\St_k$ be the Steinberg
representation of $G$, def\/ined over an arbitrary f\/ield $k$; see~\cite{St57}.
We shall be concerned here with the case where $\St_k$ is reducible. There
is only very little general knowledge about the structure of $\St_k$ in this
case; see, e.g.,~\cite{myst,gow} and the references there. Using
the theory of Gelfand--Graev representations of~$G$, Hiss~\cite{hiss2}
showed that~$\St_k$ always has a certain distinguished composition factor
with multiplicity~$1$. It appears to be extremely dif\/f\/icult to determine
further properties of this composition factor, e.g., its dimension. The
purpose of this note is to show that this composition factor can be
def\/ined in a somewhat more intrinsic way through a version of James'
submodule theorem~\cite{Jam1}; see Remark~\ref{rem2}.

\section[Groups with a split $BN$-pair]{Groups with a split $\boldsymbol{BN}$-pair} \label{sec1}

Let $G$ be a f\/inite group and $B,N\subseteq G$ be subgroups which satisfy
the axioms for an ``algebraic group with a split $BN$-pair'' in
\cite[Section~2.5]{Ca2}. We just recall explicitly those properties of $G$ which
will be important for us in the sequel. Firstly, there is a prime number
$p$ such that we have a semidirect product decomposition $B=U\rtimes H$
where $H=B\cap N$ is an abelian group of order prime to $p$ and $U$ is a
normal $p$-subgroup of $B$. The group $H$ is normal in $N$ and $W=N/H$ is
a f\/inite Coxeter group with a canonically def\/ined generating set $S$; let
$l\colon W\rightarrow \N_0$ be the corresponding length function. For each
$w\in W$, let $n_w\in N$ be such that $Hn_w=w$. Let $w_0\in W$ be the
unique element of maximal length; we have $w_0^2=1$. Let $n_0\in N$ be
a~representative of $w_0$ and $V:=n_0^{-1}Un_0$; then $U\cap V=H$. For
$w\in W$, let $U_w:=U\cap n_w^{-1}Vn_w$. (Note that $V$, $U_w$ do not
depend on the choices of $n_0$, $n_w$.) Then we have the following sharp
form of the Bruhat decomposition:
\begin{gather*}G=\coprod_{w\in W} Bn_wU_w,\qquad \mbox{with uniqueness of expressions},\end{gather*}
that is, every $g\in Bn_wB$ can be uniquely written as $g=bn_wu$ where $
b\in B$ and $u\in U_w$. Note that, since $w_0^2=1$, we have $U_{w_0}=U$
and $Bn_0B=Bn_0U=Un_0B$, in both cases with uniqueness of expressions.

Let $k$ be any f\/ield and $kG$ be the group algebra of $G$. All our
$kG$-modules will be f\/inite-dimensional and left modules.

\begin{rem} \label{bil1} Let $\fb:=\sum\limits_{b\in B} b\in kG$. Then $kG\fb$
is a left $kG$-module which is canonically isomorphic to the induced
module $\Ind_B^G(k_B)$; here, $k_B$ denotes the trivial $kB$-module. Now
$kG\fb$ carries a natural symmetric bilinear form $\langle \;,\;\rangle
\colon kG\fb\times kG\fb\rightarrow k$ such that, for any $g,g'\in G$, we
have
\begin{gather*}\langle g\fb,g'\fb\rangle= \begin{cases} 1, & \mbox{if
$gB=g'B$},\\ 0, & \mbox{otherwise}.\end{cases}\end{gather*}
This form is non-singular and $G$-invariant. For any subset $X\subseteq
kG\fb$, we denote
\begin{gather*}X^\perp:=\{a\in kG\fb\,|\, \langle a,x\rangle=0 \mbox{ for all
$x\in X$}\}.\end{gather*}
If $X$ is a $kG$-submodule of $kG\fb$, then so is $X^\perp$.
\end{rem}

\begin{rem} \label{gelgr} Let $\sigma \colon U\rightarrow k^\times$ be a
group homomorphism and set
\begin{gather*}\fu_\sigma =\sum_{u\in U} \sigma(u)u \in kG.\end{gather*}
Then $\Gamma_\sigma:= kG\fu_\sigma$ is a left $kG$-module which is isomorphic
to the induced module $\Ind_U^G(k_\sigma)$, where $k_\sigma$ denotes the
$1$-dimensional $kU$-module corresponding to~$\sigma$. Note that
$\fu_\sigma^2=|U|\fu_\sigma$. Hence, if $|U|1_k\neq 0$, then $\Gamma_\sigma$
is a projective $kG$-module.
\end{rem}

We say that $\sigma$ is non-degenerate if the restriction of $\sigma$ to
$U_s$ is non-trivial for every $s\in S$. Note that this can only occur
if $|U|1_k\neq 0$. In the case where $G=\GL_n(q)$, the following result is
contained in James \cite[Theorem~11.7(ii)]{Jam1}; see also Szechtman
\cite[Note~4.9, Section~10]{Sz1}.

\begin{lem} \label{nlem1} Assume that $|U|1_k\neq 0$ and $\sigma$ is
non-degenerate. Then the subspace $\fu_\sigma kG \fb$ is $1$-dimensional
and spanned by $\fu_\sigma n_0\fb$. Furthermore, $\fu_\sigma n_w\fb=0$
for all $w\in W$ such that $w\neq w_0$.
\end{lem}

\begin{proof} By the Bruhat decomposition, we can write any $g\in G$ in
the form $g=un_wb$ where $u\in U$, $w\in W$ and $b\in B$. Now note that
$b\fb=\fb$ for all $b\in B$ and $\fu_\sigma u=\sigma(u)^{-1} \fu_\sigma$
for all $u\in U$. Thus, $\fu_\sigma kG \fb$ is spanned by $\{\fu_\sigma n_w
\fb \,|\, w\in W\}$. Now let $w\in W$ be such that $w\neq w_0$. We shall
show that $\fu_\sigma n_w \fb=0$. For this purpose, we use the
factorisation $U=U_wU_{w_0w}$ where $U_w\cap U_{w_0w}=\{1\}$; see
\cite[Proposition~2.5.12]{Ca2}. Since $w\neq w_0$, there exists some $s\in S$ such
that $l(ws)>l(w)$ and so $U_s\subseteq U_{w_0w}$; see \cite[Propositions~2.5.7(i) and~2.5.10(i)]{Ca2}. (Note that $U_s$ is denoted by~$X_i$ in
[{\em loc.\ cit.}].) Since $n_{w^{-1}}\in n_w^{-1}H$, we obtain
\begin{gather*}\fu_\sigma n_{w^{-1}}\fb=\fu_\sigma n_w^{-1}\fb=\bigg(\sum_{u_1\in U_w}
\sigma(u_1)u_1\bigg)n_w^{-1} \bigg(\sum_{u_2\in U_{w_0w}}\sigma(u_2)
n_wu_2n_w^{-1} \fb\bigg).\end{gather*}
By the def\/inition of $U_{w_0w}$, we have $n_wU_{w_0w}n_w^{-1}\subseteq U$
and so $n_wu_2n_w^{-1}\fb=\fb$ for every f\/ixed $u_2\in U_{w_0w}$. Hence,
we obtain
\begin{gather*} \sum_{u_2\in U_{w_0w}}\sigma(u_2) n_wu_2n_w^{-1}\fb=
\bigg(\sum_{u_2\in U_{w_0w}}\sigma(u_2)\bigg)\fb.\end{gather*}
Finally, since $U_s\subseteq U_{w_0w}$ and the restriction of $\sigma$
to $U_s$ is non-trivial, the above sum evaluates to $0$. Thus, $\fu_\sigma
n_{w^{-1}}\fb=0$ for all $w\neq w_0$. Since $w_0=w_0^{-1}$, this also
implies that $\fu_\sigma n_w\fb=0$ for all $w\neq w_0$, as required.

Hence, $\fu_\sigma kG \fb$ is spanned by $\fu_\sigma n_0 \fb$. Finally,
by the sharp form of the Bruhat decomposition, every element of $Bn_0B$
has a unique expression of the form $un_0b$ where $u\in U$ and $b\in B$.
In particular, $\fu_\sigma n_0 \fb\neq 0$ and so $\dim\fu_\sigma kG\fb=1$.
\end{proof}

\begin{cor} \label{ncor1} Assume that $|U|1_k\neq 0$ and $\sigma$ is
non-degenerate. Then the map $\varphi \colon \Gamma_\sigma\rightarrow
kG\fb$, $\gamma \mapsto \gamma n_0\fb$, is a non-zero $kG$-module
homomorphism and every homomorphism $\Gamma_\sigma\rightarrow kG\fb$ is a~scalar multiple of $\varphi$.
\end{cor}

\begin{proof} The fact that $\varphi$, as def\/ined above, is a $kG$-module
homomorphism is clear; it is non-zero since $\varphi(\fu_\sigma)=\fu_\sigma
n_0\fb\neq 0$ by Lemma~\ref{nlem1}. Since $\fu_\sigma$ is a non-zero scalar
multiple of an idempotent (see Remark~\ref{gelgr}), we have $\Hom_{kG}
(\Gamma_\sigma, kG\fb) \cong \fu_\sigma kG\fb$ and this is
$1$-dimensional by Lemma~\ref{nlem1}.
\end{proof}

\section{The submodule theorem} \label{sec2}
We keep the notation of the previous section and assume throughout that
$|U|1_k\neq 0$. For any group homomorphism $\sigma\colon U \rightarrow
k^\times$, we denote by $\sigma^* \colon U\rightarrow k^\times$ the
group homomorphism given by $\sigma^*(u)= \sigma(u)^{-1}$ for all
$u\in U$. Note that, if $\sigma$ is non-degenerate, then so is $\sigma^*$.
We can now state the following version of James's ``submodule theorem''
\cite[Theorem~11.2]{Jam1}, \cite[11.12(ii)]{Jam1}.

\begin{prop}[cf.\ James \protect{\cite[Theorem~11.2]{Jam1}}] \label{othm1} Let
$\sigma\colon U\rightarrow k^\times$ be non-degenerate and consider the
submodule $\cS_\sigma:=kG\fu_\sigma n_0 \fb\subseteq kG\fb$. Then the
following hold.
\begin{itemize}\itemsep=0pt
\item[{\rm (i)}] If $M\subseteq kG\fb$ is any submodule, then either
$\cS_\sigma\subseteq M$ or $M\subseteq \cS_{\sigma^*}^\perp$.
\item[{\rm (ii)}] We have $\cS_\sigma \not\subseteq \cS_{\sigma^*}^\perp$
and $\cS_\sigma \cap \cS_{\sigma^*}^\perp \subsetneqq \cS_\sigma$ is the
unique maximal submodule of $\cS_\sigma$.
\item[{\rm (iii)}] The $kG$-module $D_\sigma:=\cS_\sigma/(\cS_\sigma
\cap \cS_{\sigma^*}^\perp)$ is absolutely irreducible and isomorphic
to the contragredient dual of $D_{\sigma^*}$.
\item[{\rm (iv)}] $D_\sigma$ occurs with multiplicity~$1$ as a
composition factor of $kG\fb$ and of $\cS_\sigma$.
\end{itemize}
\end{prop}

\begin{proof} Having established Lemma~\ref{nlem1} and Corollary~\ref{ncor1},
this readily follows from the general results in \cite[Chapter~11]{Jam1};
the only dif\/ference is that James also assumes that $\cS_\sigma=
\cS_{\sigma^*}$. As our notation and setting are somewhat dif\/ferent from
those in \cite{Jam1}, we recall the most important steps of the argument.

(i) Since $M\subseteq kG\fb$, it follows from Lemma~\ref{nlem1} that, for
any $m\in M$, there exists some $c_m\in k$ such that $\fu_\sigma m=c_m
\fu_\sigma n_0\fb$. If there exists some $m\in M$ with $c_m\neq 0$, then
$\fu_\sigma n_0\fb= c_m^{-1} \fu_\sigma m\in M$ and, hence, we have
$\cS_\sigma\subseteq M$ in this case. Now assume that $c_m=0$ for all
$m\in M$; that is, we have $\fu_\sigma M=\{0\}$. Let $m\in M$ and
$g\in G$. Using the def\/inition of $\fu_\sigma$, $\fu_{\sigma^*}$ and the
$G$-invariance of $\langle\;,\; \rangle$, we obtain
\begin{gather*} \langle m,g\fu_{\sigma^*}n_0\fb\rangle=\big\langle g^{-1}m,\fu_{\sigma^*}
n_0\fb\big\rangle=\big\langle \fu_{\sigma}\big(g^{-1}m\big),n_0\fb\big\rangle=0,\end{gather*}
where the last equality holds since $\fu_\sigma M=\{0\}$.
Thus, $M \subseteq \cS_{\sigma^*}^\perp$ in this case.

(ii) For $u,u'\in U$, we have $un_0B=u'n_0B$ if and only if $u=u'$, by
the sharp form of the Bruhat decomposition. Thus, we obtain
\begin{gather*}\langle \fu_\sigma n_0\fb,\fu_{\sigma^*}n_0\fb\rangle=
\sum_{u,u'\in U} \sigma(u)\sigma^*(u')\langle un_0\fb,u'n_0\fb\rangle
=|U|1_k\neq 0,\end{gather*}
which means that $\cS_\sigma \not\subseteq
\cS_{\sigma^*}^\perp$ and so $\cS_\sigma \cap \cS_{\sigma^*}^\perp
\subsetneqq \cS_\sigma$. Now let $M\subsetneqq \cS_\sigma$ be any maximal
submodule. Then (i) immediately implies that $M=\cS_\sigma \cap
\cS_{\sigma^*}^\perp$.

(iii) By (ii), we already know that $D_\sigma$ is irreducible.
The remaining assertions then follow exactly as in the proof of
\cite[Theorem~11.2]{Jam1}.

(iv) By construction, $D_\sigma$ occurs at least once in $kG\fb$ and in
$\cS_\sigma$. Let $\varphi\colon \Gamma_\sigma \rightarrow kG\fb$ be as
in Corollary~\ref{ncor1}. Since $\Gamma_\sigma$ is projective and
$\varphi(\Gamma_\sigma)=\cS_\sigma$, some indecomposable direct summand
of $\Gamma_\sigma$ is a projective cover of $D_\sigma$ and so the desired
multiplicity of $D_\sigma$ is at most $\dim \Hom_{kG}(\Gamma_\sigma,
kG\fb)=1$, where the last equality holds by Corollary~\ref{ncor1}.
\end{proof}

We now relate the modules $\cS_\sigma$, $D_\sigma$ in
Proposition~\ref{othm1} to the Steinberg module $\St_k$ of $G$, as
def\/ined in \cite{St57}. Recall that
\begin{gather*}\St_k=kG\fe\subseteq kG\fb,\qquad \mbox{where}\quad \fe:=
\sum_{w\in W} (-1)^{l(w)}n_w\fb.\end{gather*}
We have $\dim \St_k=|U|$; a basis of $\St_k$ is given by $\{u\fe\,|\, u\in U\}$.

\begin{prop} \label{rem21} We have $\cS_\sigma=kG\fu_\sigma \fe
\subseteq \St_k$. Consequently, $D_\sigma$ is a composition factor
$($with multiplicity~$1)$ of $\St_k$.
\end{prop}

\begin{proof} By Lemma~\ref{nlem1}, we have the identity
\begin{gather*}\fu_\sigma\fe=\sum_{w\in W} (-1)^{l(w)}\fu_\sigma
n_w\fb=(-1)^{l(w_0)} \fu_\sigma n_0\fb\end{gather*}
and so $\cS_\sigma=kG\fu_\sigma n_0\fb=kG\fu_\sigma\fe\subseteq kG\fe=\St_k$.
The statement about $D_\sigma$ then follows from Proposition~\ref{othm1}(iv).
\end{proof}

\begin{rem} \label{remgow} Gow \cite[Section~2]{gow} gives an explicit formula
for the restriction of the bilinear form $\langle \;,\;\rangle\colon kG\fb
\times kG\fb \rightarrow k$ (see Remark~\ref{bil1}) to $\St_k$. For this
purpose, he f\/irst works over $\Z$ and then rescales to obtain a non-zero
form over~$k$. One can also proceed directly, as follows. We have
\begin{gather*}\langle u_1\fe,u_2\fe\rangle=c_W\big(u_1^{-1}u_2\big)1_k \qquad \mbox{for any
$u_1,u_2\in U$},\end{gather*}
where $c_W(u):=|\{w\in W\,|\, n_w^{-1}un_w\in U\}|$ for $u\in U$.
Indeed, since $\langle \;,\;\rangle$ is $G$-invariant, it is enough to
show that $\langle \fe,u\fe\rangle=c_W(u)1_k$ for $u\in U$. Now, we have
\begin{gather*}\langle \fe,u\fe\rangle=\sum_{w,w'\in W} (-1)^{l(w)+l(w')} \langle n_w\fb,
un_{w'}\fb\rangle=\sum_{w\in W} \langle n_w\fb, un_w\fb\rangle,\end{gather*}
where the second equality holds since $n_wB=un_{w'}B$ if and only if $w=w'$.
Furthermore, $n_wB=un_wB$ if and only if $n_w^{-1}un_w\in B$. Since
$n_w^{-1}un_w$ is a $p$-element, the latter condition is equivalent to
$n_w^{-1}un_w\in U$. This yields the desired formula.
\end{rem}

\begin{rem} \label{rem2}
Since $|U|1_k\neq 0$, the module $\Gamma_\sigma$ in Remark~\ref{gelgr}
is projective. Also note that $\cS_\sigma=\varphi(\Gamma_\sigma)$
where $\varphi$ is def\/ined in Corollary~\ref{ncor1}. Using also
Proposition~\ref{rem21}, we conclude that
$\dim \Hom_{kG}(\Gamma_\sigma,\St_k)=\dim \Hom_{kG}(\Gamma_\sigma,kG
\fb)=1$.
So there is a unique indecomposable direct summand $P_{\St}$ of~$\Gamma_\sigma$ such that
\begin{gather*}\Hom_{kG}(P_{\St},\St_k)\neq 0.\end{gather*}
Being projective indecomposable, $P_{\St}$ has a unique simple quotient
whose Brauer character is denoted by $\sigma_G$ by Hiss \cite[Section~6]{hiss2}.
By Proposition~\ref{othm1}(iv) (and its proof), we now see that~$\sigma_G$
is the Brauer character of $D_\sigma$.
\end{rem}

\begin{rem} \label{socle} It is known that the socle of $\St_k$ is simple;
see \cite{myst,tin1}. We claim that this simple socle is contained
in $\cS_\sigma \cap \cS_{\sigma^*}^\perp$, unless $\St_k$ is irreducible.
Indeed, assume that $\cS_\sigma\cap \cS_{\sigma^*}^\perp=\{0\}$.
Then $D_\sigma\subseteq \St_k\subseteq kG\fb$ and, hence,
$D_\sigma$ belongs to the Harish-Chandra series def\/ined by the pair
$(\varnothing, k_H)$ (notation of Hiss \cite[Theorem~5.8]{hiss2}). By
Remark~\ref{rem2} and the argument in \cite[Lemma~6.2]{hiss2}, it then
follows that $[G:B]1_k\neq 0$. So $\St_k$ is irreducible by \cite{St57}.
\end{rem}

\section{Examples} \label{sec3}

We keep the setting of the previous section. We also assume now that $G$
is a true f\/inite group of Lie type, as in \cite[Section~1.18]{Ca2}. Thus,
using the notation in [{\em loc.\ cit.}], we have $G=\bG^F$ where~$\bG$
is a connected reductive algebraic group $\bG$ over $\overline{\F}_p$
and $F\colon \bG \rightarrow \bG$ is an endomorphism such that some power
of $F$ is a Frobenius map. Then the ingredients of the $BN$-pair in $G$
will also be derived from $\bG$: we have $B=\bB^F$ where $\bB$ is an
$F$-stable Borel subgroup of $\bG$ and $H=\bT^F$ where $\bT$ is an
$F$-stable maximal torus contained in $\bB$; furthermore, $N=N_{\bG}
(\bT)^F$ and $U=\bU^F$ where $\bU$ is the unipotent radical of $\bB$.
In this setting, one can single out a certain class of non-degenerate group
homomorphisms $\sigma\colon U \rightarrow k^\times$, as follows.

The commutator subgroup $[\bU,\bU]$ is an $F$-stable closed connected
normal subgroup of $\bU$. We def\/ine the subgroup $U^*:=[\bU,\bU]^F
\subseteq U$; then $[U,U] \subseteq U^*$. Furthermore, we shall f\/ix a~group homomorphism $\sigma \colon U\rightarrow k^\times$ which is a
{\em regular character}, that is, we have $U^*\subseteq \ker(\sigma)$ and
the restriction of $\sigma$ to~$U_s$ is non-trivial for all $s\in S$.
(Such characters always exist; see \cite[Section~8.1]{Ca2} and
\cite[Def\/inition~14.27]{DiMi}.) Then the corresponding module $\Gamma_\sigma=
kG\fu_\sigma$ is called a~{\em Gelfand--Graev module} for~$G$.
Let $h \in H$ and $\sigma^h\colon U \rightarrow k^\times$ be def\/ined by
$\sigma^h(u):=\sigma(h^{-1}uh)$ for $u \in U$. Then $\sigma^h$ also is a~regular character and
\begin{gather*} \fu_{\sigma^h}=\sum_{u \in U} \sigma^h(u)u=h\fu_\sigma h^{-1} \qquad
\mbox{for all $h \in H$}.\end{gather*}
Hence, right multiplication by $h^{-1}$ def\/ines an isomorphism between the corresponding Gelfand--Graev modules $\Gamma_\sigma$ and~$\Gamma_{\sigma^h}$.

\begin{rem} \label{rem3} Let $Z(G)$ be the center of $G$. Then $Z(G)
\subseteq H$ and $Z(G)=Z(\bG)^F$; see \cite[Proposition~3.6.8]{Ca2}. Assume now that
$Z(\bG)$ is connected. Then there are precisely $|H/Z(G)|$ regular
characters and they are all conjugate under the action of $H$; see
\cite[Proposition~8.1.2]{Ca2}. For any $h \in H$ we have
\begin{gather*} \fu_{\sigma^h}n_0\fb=h\fu_\sigma h^{-1} n_0\fb=h\fu_\sigma n_0h^{-1}\fb=
h\fu_\sigma n_0\fb\end{gather*}
and so $\cS_\sigma=\cS_{\sigma^h}$. It follows that $\cS_{\sigma}=
\cS_{\sigma^*}=\cS_{\sigma'}$ for all regular characters $\sigma$, $\sigma'$
of~$U$, and we can denote this submodule simply by $\cS_0$. By
Proposition~\ref{othm1}(iii), the simple module $D_0:=\cS_0/(\cS\cap
\cS_0^\perp)$ is now self-dual. Furthermore, we have
\begin{gather*}\dim \cS_0 \geq |H/Z(G)|.\end{gather*}
Indeed, we have $\fu_\sigma n_0\fb\in \cS_0$ for all regular characters
$\sigma$. Since pairwise distinct group homomorphisms $U\rightarrow k^\times$
are linearly independent, the elements $\fu_\sigma n_0\fb$ (where $\sigma$
runs over all regular characters of $U$) form a set of $|H/Z(G)|$ linearly
independent elements in $\cS_0$.
\end{rem}

\begin{exmp} \label{nexp1} Let $G=\GL_n(q)$ where $q$ is a prime power.
Then our module $\cS_0$ is $S_{\lambda}$ in James' notation
\cite[Def\/inition~11.11]{Jam1}, where $\lambda$ is the partition of $n$ with
all parts equal to~$1$. We claim that $\cS_0=\St_k$ in this case.

Indeed, by \cite[Theorem~16.5]{Jam1}, $\dim \cS_0$ is independent of the
f\/ield~$k$, as long as $\mbox{char}(k)\neq p$. Since $\St_\Q$ is irreducible,
we conclude that $\dim \cS_0=\dim \St_k$ and the claim follows.
Consequently, by Proposition~\ref{othm1}(ii), $D_0$ is the unique simple
quotient of $\St_k$. (The facts that $\cS_0=\St_k$ and that this module
has a unique simple quotient are also shown by Szechtman \cite[Section~4]{Sz1}.)
However, $\dim D_0$ may certainly vary as the f\/ield $k$ varies; see
the tables in \cite[p.~107]{Jam1}.
\end{exmp}

See \cite[4.14]{myst} for further examples where $\cS_0=\St_k$. On the
other hand, Gow \cite[Section~5]{gow} gives examples (where $G=\operatorname{Sp}_4(q)$)
where $\St_k$ does not have a unique simple quotient, and so $\cS_0
\subsetneqq \St_k$. Here is a further example.

\begin{exmp} \label{bnrk1} Let $G=\operatorname{Ree}\big(q^2\big)$ be the Ree group
of type ${^2G}_2$, where $q$ is an odd power of $\sqrt{3}$. Then $G$ has
a $BN$-pair of rank~$1$ and so $[G:B]=\dim \St_k+1$. Let $k$ be a
f\/ield of characteristic~$2$. Then $kG\fb$ and $\St_k$ have socle series
as follows:
\begin{gather*} kG\fb\colon \quad \begin{matrix} & k_G &\\& \varphi_2 & \\ \varphi_4&
\varphi_3& \varphi_5,\\ & \varphi_2 & \\ & k_G& \end{matrix}\qquad\qquad\qquad
\St_k\colon\quad \begin{matrix} & \varphi_2 & \\ \varphi_4&
\varphi_3& \varphi_5.\\ & \varphi_2 & \\ & k_G& \end{matrix}\end{gather*}
Here, $\varphi_i$ ($i=1,2,3,4,5$) are simple $kG$-modules and $\varphi_4$
is the contragredient dual of $\varphi_5$; see Landrock--Michler
\cite[Proposition~3.8(b)]{LaMi}. By Proposition~\ref{othm1}, we have
$D_0\cong \varphi_3$ and $\cS_0$ is the uniserial submodule with
composition factors $k_G$, $\varphi_2$, $\varphi_3$.
\end{exmp}

\pdfbookmark[1]{References}{ref}
\LastPageEnding

\end{document}